\documentclass[a4paper,10pt,oneside]{article}

\usepackage{latexsym,amsfonts,amsmath,amssymb,amsthm,pstricks-add}

\newcommand{\Z}{\mathbb{Z}}
\newcommand{\pres}[2]{\langle {#1}\ |\ {#2} \rangle}
\newcommand{\gpres}[1]{\langle {#1}\rangle}
\newcommand{\npres}[1]{\langle\langle {#1}\rangle\rangle}
\newtheorem{theorem}{Theorem}
\newtheorem{lemma}[theorem]{Lemma}

\newtheorem{maintheorem}{Theorem}

\begin{document}
\title{Fibonacci type presentations and 3-manifolds}%
\author{James Howie and Gerald Williams\thanks{Part of this research was carried out during a visit by the second named author to the Department of Mathematics at Heriot Watt
University in January and February~2016. That visit was financed by the Edinburgh Mathematical Society Research Support Fund. The authors
would like to thank the EMS for its support and the second named author would like to thank the Department of Mathematics at Heriot Watt University for its hospitality during that visit. }}

\maketitle

\begin{abstract}
We study the cyclic presentations with relators of the form $x_ix_{i+m}x_{i+k}^{-1}$ and the groups they define. These ``groups of Fibonacci type'' were introduced by Johnson and Mawdesley and they generalize the Fibonacci groups $F(2,n)$ and the Sieradski groups $S(2,n)$. With the exception of two groups, we classify when these groups are fundamental groups of 3-manifolds, and it turns out that only Fibonacci, Sieradski, and cyclic groups arise. Using this classification, we completely classify the presentations that are spines of 3-manifolds, answering a question of Cavicchioli, Hegenbarth, and Repov\v{s}. When $n$ is even the groups $F(2,n),S(2,n)$ admit alternative cyclic presentations on $n/2$ generators. We show that these alternative presentations also arise as spines of 3-manifolds.
\end{abstract}

\medskip


\noindent \textbf{MSCs:} 20F05, 57M05 (primary), 57M50 (secondary).\\

\section{Introduction}\label{sec:intro}

Let $F_n$ be the free group of rank~$n\geq 1$ with generators $x_0,\ldots , x_{n-1}$ and let $w=w(x_0,\ldots ,x_{n-1})$ be a word in $F_n$ and let $\theta: F_n\rightarrow F_n$ be an automorphism given by $x_i\mapsto x_{i+1}$ (subscripts mod~$n$). The presentation
\begin{alignat}{1}
\mathcal{G}_n(w)=\pres{x_0,\ldots ,x_{n-1}}{\theta^i(w)\ (0\leq i\leq n-1)}\label{eq:cyclicpres}
\end{alignat}
is called a \em cyclic presentation\em, and the group $G_n(w)$ that it defines is a \em cyclically presented group. \em
Since cyclic presentations are balanced (that is, they have an equal number of generators and relations) the questions as to which cyclic presentations are spines of 3-manifolds and which cyclically presented groups are fundamental groups of 3-manifolds arise and have been considered by many authors~(see, for example, \cite{CHR}). For instance, in~\cite{Dunwoody}, Dunwoody provides an algorithm to determine if a cyclic presentation is the spine of a 3-manifold. There is a large body of literature showing that certain cyclic presentations arise as spines of 3-manifolds, which hence define fundamental groups of 3-manifolds (see, for example, \cite{CRSTopProp},\cite{KimKimGenDun} and the references therein). Conversely, many theorems show that certain cyclic presentations do not define the fundamental groups of hyperbolic 3-orbifolds (in particular 3-manifolds) of finite volume (see, for example, Theorem~3.1 of any of \cite{Maclachlan},\cite{BV},\cite{CRS},\cite{CRSTopProp}, or see~\cite{SV}). We contribute to this theory by studying one three-parameter family in detail, namely the cyclic presentations $\mathcal{G}_n(m,k)=\mathcal{G}_n(x_0x_mx_k^{-1})$ (where $n\geq 1$, $0\leq m,k\leq n-1$) and the groups $G_n(m,k)$ they define.

The groups $G_n(m,k)$ are known as the \em groups of Fibonacci type \em $G_n(m,k)$ and were introduced in~\cite{JohnsonMawdesley}. Their presentations $\mathcal{G}_n(m,k)$ generalize the presentations $\mathcal{F}(2,n)=\mathcal{G}_n(1,2)$ of the \em Fibonacci groups \em $F(2,n)$ and the presentations $\mathcal{S}(2,n)=\mathcal{G}_n(2,1)$ of the \em Sieradski groups \em $S(2,n)$ of~\cite{Sieradski} and the presentations $\mathcal{H}(n,m)=\mathcal{G}_n(m,1)$ of the groups $H(n,m)$ studied in~\cite{GilbertHowie}. Together with the cyclic presentations $\mathcal{G}_n(x_0x_kx_l)$ studied in~\cite{CRS},\cite{EdjvetWilliams} they form the triangular cyclic presentations -- that is, cyclic presentations where the relators have length three. The groups $G_n(m,k)$ have been studied for their algebraic properties in~\cite{BV},\cite{COS},\cite{GilbertHowie},\cite{HowieWilliams},\cite{JohnsonMawdesley},\cite{W},\cite{W2} -- see~\cite{WilliamsRevisited} for a survey of such results. In this article we study geometric and topological aspects of the groups $G_n(m,k)$ and their presentations.

Our starting point is two results concerning the Fibonacci and Sieradski cases. The first is that if $n\geq 2$ is even then
$\mathcal{F}(2,n)$ is the spine of a closed 3-manifold and hence $F(2,n)$ is the fundamental group of a closed 3-manifold~\cite{HMK},\cite{HLM1},\cite{HLM2},\cite{CavicchioliSpaggiari} -- see Theorem~\ref{thm:evenFibonacciSpine}. The second is that if $n\geq 2$ then the presentation $\mathcal{S}(2,n)$ is the spine of a closed 3-manifold and hence $S(2,n)$ is the fundamental group of a closed 3-manifold~\cite{Sieradski},\cite{CHK} -- see Theorem~\ref{thm:SieradskiSpine}. Motivated by these results, in~\cite[Problem~6]{CHR} Cavicchioli, Hegenbarth and Repov\v{s} asked which presentations $\mathcal{G}_n(m,k)$ are spines of closed 3-manifolds. Related to this is the question as to which groups $G_n(m,k)$ are fundamental groups of closed 3-manifolds. With the exception of the (known to be challenging) groups $H(9,4),H(9,7)$ we answer the second question in Theorem~\ref{mainthm:CHRmanifoldgroup}, and this allows us to completely answer the first question in Theorem~\ref{mainthm:CHRspines}.

These theorems show that the only cases where spines are possible are in the cases of the Fibonacci and Sieradski presentations of Theorems~\ref{thm:evenFibonacciSpine},\ref{thm:SieradskiSpine}, and the only cases where the group is the fundamental group of a 3-manifold are the cases of the corresponding Fibonacci and Sieradski groups and when it is a finite cyclic group. Contained within this classification is the result (Theorem~\ref{thm:F2oddNot3mfdgp}) that for odd $n\geq 3$ the group $F(2,n)$ is the fundamental group of a 3-manifold only when $n=3,5,$ or $7$ (in which case it is cyclic). It has previously been observed (\cite[page~55]{CavicchioliSpaggiari}) that the presentation $\mathcal{F}(2,3)$ is the spine of a closed 3-manifold and the presentations $\mathcal{F}(2,5),\mathcal{F}(2,7)$ are not spines of 3-manifolds and so we can conclude that $\mathcal{F}(2,n)$ is the spine of a manifold if and only if $n=3$ or $n$ is even, confirming the expectation expressed in~\cite[page~55]{CavicchioliSpaggiari}.

An alternative presentation of the Fibonacci group $F(2,2m)$ was obtained in~\cite[Example~1.2]{KimVesnin97}, namely the presentation $\mathcal{G}_m(x_0^{-1}x_{1}^{2}x_{2}^{-1}x_{1})$; in the same way we can obtain the alternative presentation $\mathcal{G}_m(x_0x_1^2x_2x_1^{-1})$ of the Sieradski group $S(2,2m)$. In Theorem~\ref{mainthm:alternativepresspines} we prove that these alternative presentations are also spines of 3-manifolds.

\section{3-manifold groups}

By a \em 3-manifold group \em we mean the fundamental group of a (not necessarily closed, compact, or orientable) 3-manifold. A 2-dimensional subpolyhedron $P$ of a compact, connected, 3-manifold with boundary $M$ is called a \em spine \em of $M$ if $M$ collapses to $P$. A 2-dimensional subpolyhedron $P$ of a closed, connected, 3-manifold $M$ is called a \em spine \em of $M$ if $M\backslash \mathrm{Int} B^3$ collapses to $P$, where $B^3$ is a 3-ball in~$M$. (See, for example, \cite{Matveev07}.)
We shall say that a finite presentation $\mathcal{P}$ is a \em 3-manifold spine \em if its presentation complex $K_\mathcal{P}$ is a spine of a (not necessarily closed, compact, or orientable) 3-manifold. Clearly if $\mathcal{P}$ is a 3-manifold spine then the group it defines is a 3-manifold group. In this section we classify which groups $G_n(m,k)$ are 3-manifold groups.

The following theorems assert that if $n\geq 2$ is even then the Fibonacci group $F(2,n)$ is a 3-manifold group and its presentation $\mathcal{F}(2,n)$ is a 3-manifold spine and that if $n\geq 2$ then the Sieradski group $S(2,n)$ is a 3-manifold group and its presentation $\mathcal{S}(2,n)$ is a 3-manifold spine. Using new methods these theorems have recently been reproved in~\cite[Theorem~5.4]{CFLPP}, where some history of these results is also described.

\begin{theorem}[({\cite{HMK},\cite{HLM1},\cite{HLM2},\cite{CavicchioliSpaggiari}})]\label{thm:evenFibonacciSpine}
For each even $n\geq 2$ the presentation complex of $\mathcal{F}(2,n)$ is the spine of a closed 3-manifold. Moreover, it is the spine of the $n/2$-fold cyclic cover of $S^3$ branched over the figure eight knot, and this is the unique closed prime 3-manifold for which $\mathcal{F}(2,n)$ is a spine.
\end{theorem}

\begin{theorem}[(\cite{Sieradski},\cite{CHK})]\label{thm:SieradskiSpine}
For each $n\geq 2$ the presentation complex of $\mathcal{S}(2,n)$ is the spine of a closed 3-manifold. Moreover, it is the spine of the $n$-fold cyclic cover of $S^3$ branched over the trefoil knot.
\end{theorem}

By considering covers and connected sums we have that a free product $G_1*G_2$ is the fundamental group of a 3-manifold if and only if $G_1$ and $G_2$ are fundamental groups of 3-manifolds. Further, using~\cite[Theorem~1]{Heil72}, we have that $G_1*G_2$ is the fundamental group of a closed 3-manifold if and only if $G_1$ and $G_2$ are fundamental groups of closed 3-manifolds. In the case of finitely generated groups, by~\cite{Scott73}, we have that $G_1*G_2$ is the fundamental group of a compact 3-manifold if and only if $G_1$ and $G_2$ are fundamental groups of compact 3-manifolds.

By~\cite[Lemma~1.2(1)]{BV} we have that $G_n(m,k)$ is isomorphic to the free product of $d=(n,m,k)$ copies of $G_{n/d}(m/d,k/d)$ and so we may assume that $d=1$. If $(n,k)=1$ or $(n,m-k)=1$ then $G_n(m,k)\cong H(n,m')$ for some $m'$ (see~\cite[Lemma~1.3]{BV}). For this reason we can divide our theorem into a statement about the groups $H(n,m)$ and a statement about the groups $G_n(m,k)$ where the parameters satisfy $(n,m,k)=1$, $(n,k)>1$ and $(n,m-k)>1$. Note that these latter conditions imply that $n\geq 6$, $m\neq k$, $m\neq 0$ and $k\neq 0$. We prove

\begin{maintheorem}\label{mainthm:CHRmanifoldgroup}
\begin{itemize}
  \item[(a)] Suppose $n\geq 1$, $0\leq m,k\leq n-1$, $(n,m,k)=1$, $(n,k)>1$, $(n,m-k)>1$. Then $G_n(m,k)$ is a 3-manifold group if and only if $(m,k)=1$ and either $2k\equiv 0$~mod~$n$ or $2(m-k)\equiv 0$~mod~$n$, in which case $G_n(m,k)\cong \Z_s$, where $s=2^{n/2}-(-1)^{m+n/2}$.

      \item[(b)] Suppose $n\geq 1$, $(n,m)\neq (9,4)$, $(9,7)$. Then $H(n,m)$ is a 3-manifold group if and only if the pair $(n,m)$ is one of the following:
      \begin{itemize}
        \item[(i)] $(n,0)$ in which case $H(n,m)\cong \Z_{2^n-1}$;
        \item[(ii)] $(n,1)$ in which case $H(n,m)=1$;
        \item[(iii)] $(n,2)$, in which case $H(n,m)\cong S(2,n)$;
        \item[(iv)] $(n,n-1)$ ($n\geq 4$ is even or $n=3,5,$ or $7$), in which case $H(n,m)\cong F(2,n)$;
        \item[(v)] $(5,3)$ or $(7,4)$, in which case $H(n,m)\cong F(2,5)\cong \Z_{11}$ or $H(n,m)\cong F(2,7) \cong \Z_{29}$, respectively;
        \item[(vi)] $(2t-2,t)$ ($t\geq 3$) in which case $H(n,m)\cong \Z_{2^{t-1}+1}$.
      \end{itemize}
\end{itemize}
\end{maintheorem}

While we are unable to deal with the groups $H(9,4),H(9,7)$, Theorem~3.1 of~\cite{BV} gives that they are not the fundamental groups of hyperbolic 3-orbifolds (in particular, 3-manifolds) of finite volume. Further, in Lemma~\ref{lem:H94H97} we shall show that the presentations $\mathcal{H}(9,4)$, $\mathcal{H}(9,7)$ are not 3-manifold spines. It therefore seems unlikely that either $H(9,4)$ or $H(9,7)$ is a 3-manifold group. It is not known whether the groups $H(9,4),H(9,7)$ are finite or infinite, or if their presentations $\mathcal{H}(9,4),\mathcal{H}(9,7)$ are aspherical (that is, if their presentation complexes are topologically aspherical).

A group $G$ is of \em type FK \em if there is a finite aspherical cell complex $K(G,1)$ whose fundamental group is isomorphic to $G$. If $G$ is of type FK then its Euler characteristic $\chi(G)$ is defined to be the Euler characteristic of a $K(G,1)$. An immediate corollary of~\cite[Theorem~2]{Ratcliffe} is that if $G$ is of type FK and $\chi(G)\geq 1$ then $G$ is a (virtual) 3-manifold group if and only if $G=1$. In particular, if
$G=\pi_1(K_\mathcal{P})\neq 1$ where $K_\mathcal{P}$ is an aspherical presentation complex of a finite balanced presentation $\mathcal{P}$, then $G$ is not a 3-manifold group. (Moreover, since groups defined by aspherical presentations are torsion-free it also follows that $G$ is not the fundamental group of any 3-orbifold.) This is a useful tool for us as, with the exception of the presentations $\mathcal{H}(9,4),\mathcal{H}(9,7)$, the aspherical presentations $\mathcal{G}_n(m,k)$ have been classified in~\cite{GilbertHowie},\cite{W}.

The presentations $\mathcal{F}(2,n)$ are not aspherical. We now show that, nevertheless, if $n$ is odd then $F(2,n)$ is rarely a 3-manifold group.

\begin{theorem}\label{thm:F2oddNot3mfdgp}
If $n\geq 3$ is odd then $F(2,n)$ is a 3-manifold group if and only if $n=3,5,$ or~$7$.
\end{theorem}

\begin{proof}
If $n=3,5,$ or $7$ then $G=F(2,n)$ is a finite cyclic group so assume that $n\geq 9$ is odd. Let $g=x_0x_1\ldots x_{n-1}$. We have that $g^2=1$ (\cite[page~83]{Johnsonbook}), and by (the proof of)~\cite[Proposition~3.1]{BV} we have that $G/\npres{g}\cong G^\mathrm{ab}$. Now, as argued in~\cite[Proposition~3.1]{BV}, $G^\mathrm{ab}$ is finite and for $n\geq 9$ the group $G$ is infinite (see~\cite{NewmanF29},\cite{HoltF29},\cite{Chalk},\cite{ThomasRev}) so $G\not \cong G^\mathrm{ab}$ and hence $g$ is an element of $G$ of order~2.

Suppose that $G$ is the fundamental group of a 3-manifold $M$. Since $G/\npres{g}\cong G^\mathrm{ab}$ we have that $g\in [G,G]$ and hence $g$ is orientation preserving and so $\gpres{g}\cong \Z_2$ is an orientation preserving subgroup of $G=\pi_1(M)$ of finite order. Then by~\cite[Theorem~8.2]{Epstein61} (see also~\cite[Theorem~9.8]{Hempel}) we have that $M=R\# M_1$, where $R$ is closed and orientable, $\pi_1(R)$ is finite, and $\gpres{g}$ is conjugate to a subgroup of $\pi_1(R)$. Since $G$ is infinite we have that $\pi_1(M_1)\neq 1$ and since it is a 2-generator group, Gru\v{s}ko's Theorem implies that $\pi_1(R)$ and $\pi_1(M_1)$ are each cyclic. But the derived subgroup of a free product of cyclic groups is free contradicting our deduction that $g\in [G,G]$ is an element of order 2.
\end{proof}

The classification in~\cite{GilbertHowie} gives that the presentations $\mathcal{H}(6,3)$, $\mathcal{H}(7,3)$, $\mathcal{H}(7,5)$, $\mathcal{H}(8,3)$, $\mathcal{H}(9,3)$, $\mathcal{H}(9,6)$ are not aspherical. Again, we can show that the corresponding groups are not 3-manifold groups.

\begin{lemma}\label{lem:HnmNot3MFD}
The following groups are not 3-manifold groups: $H(6,3)$, $H(7,3)\cong H(7,5)$, $H(8,3)$, $H(9,3)\cong H(9,6)$.
\end{lemma}

\begin{proof}
For $H(6,3)$, $H(8,3)$ and $H(7,3)\cong H(7,5)$ we use the classification of abelian subgroups of 3-manifold groups given in \cite[Theorem~9.1]{Epstein61} (see also, for example~\cite[Theorem~9.13]{Hempel}). We have that $H(6,3)\cong \Z_2^3 \rtimes \Z_7$ (\cite{JohnsonMawdesley}) and so contains $\Z_2^3$, and so $H(6,3)$ is not a 3-manifold group . A calculation using GAP~\cite{GAP} shows that the third derived subgroup of $H(8,3)$ is isomorphic to $\Z_3^6$, and so $H(8,3)$ is not a 3-manifold group. Using quotpic R.M.Thomas showed (see~\cite[Theorem~7.3]{GilbertHowie}) that the second derived subgroup of $H(7,3)\cong H(7,5)$ is isomorphic to $\Z^8$, and so $H(7,3)\cong H(7,5)$ is not a 3-manifold group.

Consider now the group $H=H(9,3)$ (which, as noted in~\cite{GilbertHowie}, is isomorphic to $H(9,6)$). Using the van Kampen diagram given in~\cite[Figure~2]{WilliamsRevisited} we may see that, setting $g=x_0x_4x_8x_3x_7x_2x_6x_1x_5$ the relation $g^2=1$ holds in $H=H(9,3)$.
A calculation in GAP gives that $H(9,3)/\npres{g}$ is a finite group of order $2^{15}\cdot 7$, and since $H(9,3)$ is infinite~\cite[Lemma~15]{COS} we deduce that $g$ is an element of order exactly~$2$ in $H$.

Suppose for contradiction that $H$ is a 3-manifold group. Since $H$ is finitely presentable there is a compact 3-manifold $M$ with $\pi_1(M)\cong H$~(\cite{Scott73}). We have that $H^\mathrm{ab}\cong \Z_7$ and hence there is no epimorphism of $H$ onto $\Z_2$, so $M$ is orientable. Let $M=M_1\#\ldots \#M_l$ be the unique prime decomposition of $M$ corresponding to a maximal free product decomposition $H_1* \ldots *H_l$ of $H$. Since $M$ is orientable no $M_j$ contains a 2-sided projective plane so, by~\cite[Corollary~9.9]{Hempel}, each $H_j=\pi_1(M_j)$ is either finite or torsion-free. Since $H$ is infinite and has non-trivial torsion it must therefore be a non-trivial free product. Another calculation in GAP shows that $H$ has a 3-generator presentation so by Gru\v{s}ko's Theorem at least one of the free factors must be cyclic. But $H^\mathrm{ab}\cong \Z_7$ so we have that $H\cong \Z_7 * Q$, where $Q$ is perfect. Then the derived subgroup $H'$, being the free product of 7 copies of $Q$, must also be perfect, but a calculation in GAP shows that $(H')^\mathrm{ab}\cong \Z_2^6$ so $H'$ is not perfect, a contradiction.
\end{proof}

The above results, together with the asphericity classifications mentioned earlier, combine to prove Theorem~\ref{mainthm:CHRmanifoldgroup}.

\begin{proof}[Proof of Theorem~\ref{mainthm:CHRmanifoldgroup}]
(a) As described above, by~\cite[Theorem~2]{Ratcliffe} we may assume that the presentation $\mathcal{G}_n(m,k)$ is not aspherical. Theorem~2 of~\cite{W} gives that $(m,k)=1$ and either $2k\equiv 0$~mod~$n$ or $2(m-k)\equiv 0$~mod~$n$, in which case (by~\cite[Lemma~3]{W}) $G_n(m,k)\cong \Z_s$, where $s=2^{n/2}-(-1)^{m+n/2}$, and so is a 3-manifold group.

(b) If $m=1$ then $H(n,m)=1$ so we may assume $m\neq 1$. By~\cite[Theorem~2]{Ratcliffe} we may assume that the presentation $\mathcal{H}(n,m)$ is not aspherical. Then Theorem~3.2 of~\cite{GilbertHowie} gives that $(n,m)$ is one of the following:
$(n,0)$, $(n,2)$ ($n\geq 2$), $(n,n-1)$ ($n\geq 3$), $(2t-1,t)$ ($t\geq 3$), $(2t-2,t)$ ($t\geq 3$), $(6,3)$, $(7,3)$, $(7,5)$, $(8,3)$, $(9,3)$, $(9,6)$.
If $(n,m)=(n,0)$ then $H(n,m)\cong \Z_{2^n-1}$ (by~\cite[Proposition~2.2(b)]{GilbertHowie}) so it is a 3-manifold group. If $(n,m)=(n,2)$ we have $H(n,m)=S(2,n)$ so it is a 3-manifold group by Theorem~\ref{thm:SieradskiSpine}. If $(n,m)=(n,n-1)$ then $H(n,m)\cong F(2,n)$ so it is a 3-manifold group if and only if $n=3,5,$ or $7$ or $n$ is even by Theorem~\ref{thm:evenFibonacciSpine} and Theorem~\ref{thm:F2oddNot3mfdgp}. If $(n,m)=(2t-1,t)$ ($t\geq 3)$ then $H(n,m)\cong F(2,2t-1)$ so is a 3-manifold group if and only if $t=3$ or $4$ by Theorem~\ref{thm:F2oddNot3mfdgp}. If $(n,m)=(2t-2,t)$ then $H(n,m)\cong \Z_{2^{t-1}+1}$ by~\cite[Proposition~2.2(c)]{GilbertHowie}, so is a 3-manifold group. In the remaining cases $H(n,m)$ is not a 3-manifold group by Lemma~\ref{lem:HnmNot3MFD}.
\end{proof}

The groups $G_n(x_0x_mx_k^{-1})$ belong to the class of groups $G_n^\epsilon(m,k,h)$ considered in~\cite{CRS}. Problem~3.4 of that paper asks for which values of the parameters $m,k,h,\epsilon$ the group $G_n^\epsilon(m,k,h)$ is the fundamental group of a closed, connected, orientable 3-manifold for infinitely many~$n$. Theorem~\ref{mainthm:CHRmanifoldgroup} can therefore be viewed as a starting point for answering this question. The cyclically presented groups $G_n(x_0x_kx_l)$ considered in~\cite{CRS},\cite{EdjvetWilliams} also belong to the class $G_n^\epsilon(m,k,h)$ and so we record without proof the following related result. This can be obtained from the results of~\cite{EdjvetWilliams} together with arguments similar to those used above, and with the fact that non-cyclic metacyclic groups of odd order are not 3-manifold groups
 (see, for example,~\cite[page~111]{Orlik}).

\begin{theorem}\label{thm:CRS3mfdgps}
Suppose $(n,k,l)=1$. The cyclically presented group $G_n(x_0x_kx_l)$ is a 3-manifold group if and only if it is isomorphic to one of the following groups: $\Z_3$, $\Z_{2^n-(-1)^n}$, $\Z*\Z*\Z_{19}$, $\Z*\Z*\Z_{\gamma}$ ($\gamma=(2^{n/3}-(-1)^{n/3})/3$), $\Z*\Z$.
\end{theorem}

(The conditions on the parameters $n,k,l$ that identify when each group occurs can be extracted from~\cite[Table~1]{EdjvetWilliams}.)

\section{Presentations as spines of manifolds}

A wedge $P_1 \vee P_2$ of two 2-dimensional subpolyhedra of a 3-manifold $M$ is a spine of $M$ if and only if $P_1$ and $P_2$ are the spines of 3-manifolds $M_1$ and $M_2$ where $M=M_1 \# M_2$ (see~\cite[Lemma~1]{Knutson}).
The argument of~\cite[Lemma~1.2(1)]{BV} shows that the presentation complex of $\mathcal{G}_n(m,k)$ is the wedge of $d=(n,m,k)$ copies of the presentation complex of $\mathcal{G}_{n/d}(m/d,k/d)$ so again we may assume that $d=1$. Similarly, the proof of~\cite[Lemma~1.3]{BV} gives that if $(n,k)=1$ or $(n,m-k)=1$ then the presentation complex of $\mathcal{G}_n(m,k)$ is homeomorphic to the presentation complex of $\mathcal{H}(n,m')$ for some $m'$ so, as with Theorem~\ref{mainthm:CHRmanifoldgroup}, we may divide the statement into two parts. Note that in part~(b) we no longer include the hypotheses that $(n,m)\neq (9,4),(9,7)$. We use the symbol $\simeq$ to denote homeomorphism of (the presentation complexes) of presentations.

\begin{maintheorem}\label{mainthm:CHRspines}
\begin{itemize}
  \item[(a)] Suppose $n\geq 1$, $0\leq m,k\leq n-1$, $(n,m,k)=1$, $(n,k)>1$, $(n,m-k)>1$. Then $\mathcal{G}_n(m,k)$ is not a 3-manifold spine.

      \item[(b)] Suppose $n\geq 1$. Then $\mathcal{H}(n,m)$ is a 3-manifold spine if and only if one of the following holds:
      \begin{itemize}
        \item[(i)] $m=1$, in which case $\mathcal{H}(n,m)$ is a spine of~$S^3$;
        \item[(ii)] $m=2$, in which case $\mathcal{H}(n,m)\simeq \mathcal{S}(2,n)$;
        \item[(iii)] $m=n-1$ and either $n=3$ or $n$ is even, in which case $\mathcal{H}(n,m)\simeq \mathcal{F}(2,n)$.
      \end{itemize}
\end{itemize}
\end{maintheorem}

In Lemmas~\ref{lem:H94H97},\ref{lem:Hn0nonspine},\ref{lem:cyclicCHRnonspine} we shall show that certain presentations are not 3-manifold spines. In Lemma~\ref{lem:H(n,1)spine} we show that the presentation $\mathcal{H}(n,1)$ is a 3-manifold spine. The remaining presentations are either 3-manifold spines by Theorems~\ref{thm:evenFibonacciSpine},\ref{thm:SieradskiSpine}, or do not define 3-manifold groups, by Theorem~\ref{mainthm:CHRmanifoldgroup}, so are not 3-manifold spines.

Recall that the \em Whitehead graph \em of a presentation $\mathcal{P}=\pres{X}{R}$ is the graph with vertices $v_x,v_x'$ ($x \in X$) and an edge $(v_x,v_y)$ (resp.\,$(v_x',v_y')$, $(v_x,v_y')$) for each occurrence of a cyclic subword $xy^{-1}$ (resp.\,$x^{-1}y$, $(xy)^{\pm 1}$) in a relator $r\in R$. Writing $v_i,v_i'$ in place of $v_{x_i},v_{x_i}'$ the Whitehead graph of the cyclic presentation $\mathcal{G}_n(m,k)$ therefore has $2n$ vertices $v_i,v_i'$ and $3n$ edges $(v_i,v_{i+m}')$, $(v_i,v_{i+m-k})$, $(v_i',v_{i+k}')$ $(0\leq i \leq n-1$, subscripts mod~$n$).

The following information can be found in~\cite{Neuwirth},\cite[Chapter~9]{SeifertThrelfall}. Suppose that $\mathcal{P}$ is a presentation such that its presentation complex $K_{\mathcal{P}}$ is the spine of a closed 3-manifold~$M$. Then there is a 3-complex~$C$ with face pairing such that $M$ is obtained from $C$ by identifying the faces. Moreover, the resulting cell structure on $M$ has a single vertex, a single 3-cell, and 2-skeleton homeomorphic to $K_{\mathcal{P}}$.
The Whitehead graph $\Gamma$ is the link of the single vertex of $K_\mathcal{P}$ and so embeds in the link of the single vertex of $M$. The link of the vertex of $M$ is a 2-sphere and so $\Gamma$ has a planar embedding on this sphere (see also~\cite[page~33]{AngeloiMetzler}). In the case when $\Gamma$ is connected the 3-complex $C$ is a polyhedron~$\pi$ bounding a 3-ball. Given a planar embedding of $\Gamma$ there is a one-one correspondence between the faces $F$ of this embedding and the vertices $u_F$ of $\pi$. This correspondence maps a face $F$ of degree $d$ of $\Gamma$ to a vertex of degree $d$ of $\pi$. Moreover, if the vertices in the face $F$ read cyclically around the face are $w_0,\ldots , w_d$ where $w_i\in \{v_i,v_i'\}$ then the directed edges $e_0,\ldots ,e_d$ incident to $u_F$, read cyclically are labelled $x_0,\ldots , x_d$ and directed towards $u_F$ if $w_i=v_i$ and away from $u_F$ if $w_i=v_i'$. The faces of $\pi$ correspond to the relators of the presentation $\mathcal{P}$: for each relator $r$ there are two faces whose boundaries, read either clockwise or anticlockwise, spell~$r$.

The groups defined by the presentations considered in Lemmas~\ref{lem:H94H97},\ref{lem:Hn0nonspine},\ref{lem:cyclicCHRnonspine} have finite abelianisation of odd order. It this situation it follows from~\cite[Lemma~1]{OsborneStevensIII} that if the presentation complex is the spine of a 3-manifold then it is the spine of a closed 3-manifold. In these lemmas we either show that the Whitehead graph is non-planar or we show that the Whitehead graph is connected and has a unique planar embedding and that this planar embedding can be used to show that the polyhedron $\pi$ does not exist. Either way, the presentation complex is not a 3-manifold spine.

\begin{lemma}\label{lem:H94H97}
The presentations $\mathcal{H}(9,4),\mathcal{H}(9,7)$ are not 3-manifold spines.
\end{lemma}

\begin{proof}
We argue that for $m=4$ or $7$, the Whitehead graph $\Gamma$ of $\mathcal{H}(9,m)$ is non-planar.
We have that $\Gamma$ is the graph with vertices $v_0,\ldots ,v_8, v_0',\ldots ,v_8'$ and edges $(v_i,v_{i+3})$, $(v_i',v_{i+1}')$, $(v_i,v_{i+m}')$ ($0\leq i\leq 8$, subscripts mod~$9$). Removing the edges $(v_0',v_1')$, $(v_3',v_4')$, $(v_6',v_7')$ then contracting the edges $(v_i,v_{i+3})$ ($0\leq i\leq 8$), $(v_i',v_{i+1}')$ ($i=1,2,4,5,7,8$) and removing loops leaves the non-planar graph $K_{3,3}$ so $\Gamma$ is non-planar.
\end{proof}

\begin{lemma}\label{lem:Hn0nonspine}
If $n\geq 3$ then the presentation $\mathcal{H}(n,0)=\mathcal{G}_n(x_0^2x_1^{-1})$ is not a 3-manifold spine.
\end{lemma}

\begin{proof}
The Whitehead graph for $\mathcal{H}(n,0)$ has vertices $v_0,\ldots ,v_{n-1}, v_0',\ldots ,v_{n-1}'$ and edges $v_i-v_i'$, $v_i-v_{i+1}$, $v_i'-v_{i+1}'$ for each $0 \leq  i \leq n-1$ so there is a unique planar embedding with two $n$-gons, $v_0 - v_1 - \cdots - v_{n-1} -v_0$ and $v_0' - v_1' - \cdots - v_{n-1}' -v_0'$, and~$n$ 4-gons $v_i-v_{i+1}-v_{i+1}' -v_i'-v_i$ ($0\leq i \leq n-1$).
If $\mathcal{H}(n,0)$ is a 3-manifold spine then in the face pairing polyhedron $\pi$ there is a degree~4 vertex $u$ corresponding to the
4-gon $v_0-v_{1}-v_{1}'-v_0'-v_0$. The 2-cells incident to $u$ correspond to the relators $x_0^2x_1^{-1}$, $x_1^2x_2^{-1}$, $x_0^2x_1^{-1}$, $x_0^2x_1^{-1}$ and so there are more than two 2-cells for the relator $x_0^2x_1^{-1}$, a contradiction.
\end{proof}

\begin{lemma}\label{lem:cyclicCHRnonspine}
Suppose $n\geq 1$, $(m,k)=1$, $m\neq k$, $k\neq 0$ and either $2k\equiv 0$~mod~$n$ or $2(m-k)\equiv 0$~mod~$n$. Then $\mathcal{G}_n(m,k)$ is not a 3-manifold spine.
\end{lemma}

\begin{proof}
If $2(m-k)\equiv 0$~mod~$n$ then $\mathcal{G}_n(m,k)$ has relators $x_ix_{i+m}x_{i+m+n/2}^{-1}$. Inverting these and replacing each $x_i$ by its inverse, then cyclically permuting gives the relators $x_{i+m}x_ix_{i+m+n/2}^{-1}$. Subtracting $m$ from the subscripts gives $x_{i}x_{i-m}x_{i+n/2}^{-1}$; then negating the subscripts and setting $j=-i$ gives $x_{j}x_{j+m}x_{j+n/2}^{-1}$ which are the relators for the case $2k\equiv 0$~mod~$n$. Thus, without loss of generality, we may assume $k=n/2$ and $(m,n/2)=1$. Then the Whitehead graph $\Gamma$ has vertices $v_0,\ldots ,v_{n-1}, v_0',\ldots ,v_{n-1}'$ and edges $v_i-v_{i+m+n/2}$, $v_i'-v_{i+n/2}'$, $v_i-v_{i+m}'$ ($0 \leq i \leq n-1$). The hypothesis $(m,k)=1$ implies that $(m+n/2,n)=1$ or $2$.

Suppose that $(m+n/2,n)=1$. Then there exist integers $\alpha,\beta$ such that $\alpha(m+n/2)+\beta n=1$, so $\alpha(m+n/2)\equiv 1$~mod~$n$ and $\alpha$ is odd. For each $i$ ($0\leq i\leq n-1$) set $w_i=v_{(m+n/2)i}$, $w_i'=v_{(m+n/2)i}'$  so $v_i=w_{\alpha i}$ and $v_i'=w_{\alpha i}'$. For each $i$, writing $j=\alpha i$~mod~$n$ gives that $\Gamma$ has vertices $w_j,w_j'$ ($0\leq j\leq n-1$) and edges $(w_j,w_{j+1})$, $(w_j',w_{j+n/2}')$, $(w_j,w_{j+\alpha m}')$. Contracting the edges $(w_j,w_{j+\alpha m}')$ gives the graph with vertices $w_j$ ($0\leq j\leq n-1$) and edges $(w_j,w_{j+1})$, $(w_j,w_{j+n/2})$, which is non-planar for $n\geq 6$. Hence $\Gamma$ is non-planar and so $\mathcal{G}_n(m,k)$ is not a 3-manifold spine.

Suppose then that $(m+n/2,n)=2$. Then there is a planar embedding of $\Gamma$ with two $n/2$-gons ($v_0-v_{m+n/2}-v_{2m}-\cdots - v_{-2m}-v_{-m+n/2}-v_0$ and $v_{n/2}-v_{m}-v_{2m+n/2}-\cdots - v_{-2m+n/2}-v_{-m}-v_{n/2}$), $n/2$ 2-gons ($v_i'-v_{i+n/2}'-v_i'$, $0\leq i\leq n/2-1)$, and $n/2$ octagons ($v_i-v_{i+m+n/2}-v_{i+2m+n/2}'-v_{i+2m}'-v_{i+m}-v_{i+n/2}-v_{i+m+n/2}'-v_{i+m}'-v_i$, $0\leq i\leq n/2-1)$. Moreover, this planar embedding is unique up to self homeomorphism of $S^2$. If $\mathcal{G}_n(m,k)$ is a 3-manifold spine then in the face pairing polyhedron~$\pi$ there are two sink vertices of degree $n/2$ and $n/2$ source vertices of degree 2, and the remaining $n/2$ vertices are of degree 8 and are neither sources nor sinks.

At the $i$'th source vertex, $u_i$, the two outgoing edges have labels $x_i$ and $x_{i+n/2}$. The union of the two incident 2-cells is a 2-gon $C_i$ whose boundary label is the positive word $x_{i+m}x_{i+m+n/2}$. In particular, neither of the vertices of $C_i$ is a sink or a source vertex. The incoming edges at the two sink vertices, $v_{\mathrm{even}},v_{\mathrm{odd}}$, are, in cyclic order $x_0,x_{m+n/2},x_{2m},\ldots , x_{-m+n/2}$ and $x_1,x_{1+m+n/2},x_{1+2m},\ldots , x_{1-m+n/2}$, respectively. The union of the $n/2$ 2-cells incident at $v_{\mathrm{even}}$ (resp.\,$v_\mathrm{odd}$) is an $n/2$-gon $D_{\mathrm{even}}$ (resp.\,$D_{\mathrm{odd}}$) with boundary label $W_\mathrm{even} = x_mx_{2m+n/2}x_{3m}\ldots x_{-m}x_{n/2}$ (resp. $W_\mathrm{odd} = x_{1+m}x_{1+2m+n/2}x_{1+3m}\ldots x_{1-m}x_{1+n/2}$). Now $D_\mathrm{even}\bigcup D_\mathrm{odd}$ contains precisely one 2-cell corresponding to each of the relators, as does $\bigcup_i C_i$, so the 2-sphere $\partial \pi$ must be formed from $D_\mathrm{even}$, $D_\mathrm{odd}$ and the $C_i$ ($0\leq i\leq n/2-1$) by identifying similarly labelled 1-cells in their boundaries. However, identifying the 1-cells of $\partial D_\mathrm{even}$ with the corresponding 1-cells of the $\partial C_i$ makes $D_\mathrm{even} \cup \left( \bigcup_i C_i\right)$ into an $n/2$-gon whose boundary label is the \em reverse \em of $W_\mathrm{odd}$, that is $x_{1+n/2}x_{1-m}\ldots x_{1+3m}x_{1+2m+n/2}x_{1+m}$. Therefore it is not possible to form the 2-sphere $\partial \pi$ from this $n/2$-gon and $D_\mathrm{odd}$ by identifying the 1-cells in their boundaries, and hence $\mathcal{G}_n(m,k)$ is not a 3-manifold spine.
\end{proof}

\begin{lemma}\label{lem:H(n,1)spine}
Let $n\geq 1$. Then the presentation $\mathcal{H}(n,1)=\mathcal{G}_n(x_0x_1x_1^{-1})$ is a spine of~$S^3$.
\end{lemma}

\begin{proof}
If $n=1$ then $\mathcal{H}=\mathcal{H}(n,1)=\pres{x_0}{x_0x_0x_0^{-1}}$ which is a spine of $S^3$ by~\cite{Zeeman}, so assume $n\geq 2$. Consider the polyhedron in Figure~\ref{fig:H(n,1)Polyhedron} with face pairings as shown. We shall show that the identification of faces results in a 3-complex $M$ whose 2-skeleton is the presentation complex of $\mathcal{H}$, so $M$ has one $0$-cell, $n$ 1-cells, $n$ 2-cells and one $3$-cell, so is a manifold by~\cite[Theorem~I, Section~60]{SeifertThrelfall}. To this end note that all edges labelled $x_1$, namely $[S,S]_1$,$[N,S]_1$,$[S,v_1]$ are identified by the following cycle of faces identifications:
\[ [S,v_1] \stackrel{F_0}{\longleftrightarrow} [S,S]_1 \stackrel{F_1}{\longleftrightarrow} [N,S]_1 \stackrel{F_0}{\longleftrightarrow} [S,v_1]. \]
By symmetry, for each $0\leq i\leq n-1$, all edges labelled $x_i$, namely $[S,S]_i$, $[N,S]_i$, $[S,v_i]$ are similarly identified. The initial vertices of the edges labelled $x_i$, namely $N,S$ are identified and the terminal vertices of the edges labelled $x_i$, namely $S,v_i$ are identified. Therefore the resulting complex has one $0$-cell, $n$ $1$-cells, $n$ $2$-cells, and one $3$-cell, and so its 2-skeleton is the presentation complex of $\mathcal{H}$, and the resulting manifold is~$S^3$.
\end{proof}

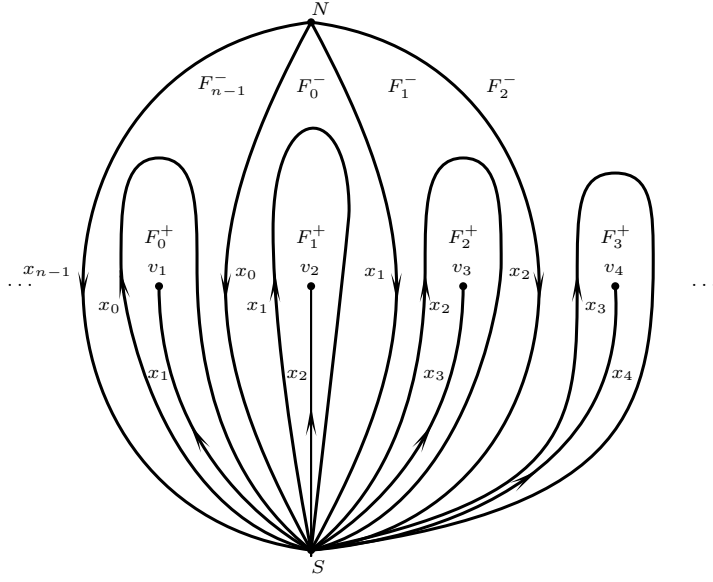
\begin{figure}
\psset{xunit=0.5cm,yunit=0.5cm,algebraic=true,dimen=middle,dotstyle=o,dotsize=5pt 0,linewidth=0.8pt,arrowsize=4pt ,arrowinset=2}
{
\begin{pspicture}(-5.5,-7.5)(7.875405,6)
\begin{scriptsize}
\psdots[dotsize=3pt,dotstyle=*,linecolor=black](6,7)
\rput[bl](6,7.2){\black{$N$}}

\psdots[dotsize=3pt,dotstyle=*,linecolor=black](6,-7)
\rput[bl](6,-7.6){\black{$S$}}
\psbezier[ArrowInside=->,linecolor=black, linewidth=0.04]
(6,7)(-2,6)(-2,-6)(6,-7)
\psbezier[ArrowInside=->,linecolor=black, linewidth=0.04]
(6,7)(3,1.4)(3,-1.5)(6,-7)
\psbezier[ArrowInside=->,linecolor=black, linewidth=0.04]
(6,7)(14,6)(14,-6)(6,-7)
\psbezier[ArrowInside=->,linecolor=black, linewidth=0.04]
(6,7)(9,1.4)(9,-1.5)(6,-7)
\psbezier[linecolor=black, linewidth=0.04]
(6,-7)(6,-7)
(5,-2)(5,1.5)
(5,5)(7,5)
(7,2)(7,1.5)
(6,-7)(6,-7)
\psbezier[linecolor=black, linewidth=0.04]
(6,-7)(2,-6)(1,-0.6)(1,0.6)(1,2)(1,3.4)(2,3.4)(3,3.4)(3,2)(3,0.6)(3,-0.6)(3,-4.6)(6,-7)
\psbezier[linecolor=black, linewidth=0.04]
(6,-7)(10,-6)(11,-0.6)(11,0.6)(11,2)(11,3.4)(10,3.4)(9,3.4)(9,2)(9,0.6)(9,-0.6)(9,-4.6)(6,-7)
\psbezier[linecolor=black, linewidth=0.04]
 (6,-7)(14,-6)(15,-3.6)(15,0.6)(15,2)(15,3.0)(14,3.0)(13,3.0)(13,2)(13,0.2)(13,-3.6)(13,-5.6)(6,-7)

\psline[ArrowInside=->](6,-7)(6,0)
\psdots[dotsize=3pt,dotstyle=*,linecolor=black](6,0)
\rput[bl](5.7,0.3){\black{$v_2$}}
\psline[ArrowInside=->](5.05,0.0)(5.05,0.1)

\rput[bl](5.6,1){\black{$F_{1}^+$}}
\rput[bl](5.6,5){\black{$F_{0}^-$}}

\psbezier[ArrowInside=->,linecolor=black, linewidth=0.04]
(6,-7)(6,-7)(2,-5)(2,0)
\psdots[dotsize=3pt,dotstyle=*,linecolor=black](2,0)
\rput[bl](1.7,0.3){\black{$v_1$}}
\rput[bl](1.6,1){\black{$F_{0}^+$}}
\rput[bl](3.0,5){\black{$F_{n-1}^-$}}
\psline[ArrowInside=->](1.04,0.0)(0.99,0.28)

\psbezier[ArrowInside=->,linecolor=black, linewidth=0.04](6,-7)(6,-7)(10,-5)(10,0)
\psdots[dotsize=3pt,dotstyle=*,linecolor=black](10,0)
\rput[bl](9.7,0.3){\black{$v_3$}}
\rput[bl](9.6,1){\black{$F_{2}^+$}}
\rput[bl](8.0,5){\black{$F_{1}^-$}}
\psline[ArrowInside=->](9.0,0.0)(9.0,0.1)

\psbezier[ArrowInside=->,linecolor=black, linewidth=0.04]
(6,-7)(6,-7)(14.5,-6.5)(14,0)
\psdots[dotsize=3pt,dotstyle=*,linecolor=black](14,0)
\rput[bl](13.7,0.3){\black{$v_4$}}
\rput[bl](13.6,1){\black{$F_{3}^+$}}
\rput[bl](10.6,5){\black{$F_{2}^-$}}
\psline[ArrowInside=->](13.0,0.0)(13.0,0.1)

\rput[bl](16,0){\black{$\ldots$}}

\rput[bl](-2,0){\black{$\ldots$}}

\rput[bl](-1.58,0.2){\black{$x_{n-1}$}}
\rput[bl](4.0,0.2){\black{$x_{0}$}}
\rput[bl](7.4,0.2){\black{$x_{1}$}}
\rput[bl](11.2,0.2){\black{$x_{2}$}}

\rput[bl](0.4,-0.7){\black{$x_{0}$}}
\rput[bl](4.3,-0.7){\black{$x_{1}$}}
\rput[bl](9.1,-0.7){\black{$x_{2}$}}
\rput[bl](13.2,-0.7){\black{$x_{3}$}}

\rput[bl](1.7,-2.5){\black{$x_1$}}
\rput[bl](5.35,-2.5){\black{$x_2$}}
\rput[bl](8.95,-2.5){\black{$x_3$}}
\rput[bl](13.9,-2.5){\black{$x_4$}}
\end{scriptsize}
\end{pspicture}}

\caption{Face pairing polyhedron for the presentation $\mathcal{H}(n,1)=\mathcal{G}_n(x_0x_1x_1^{-1})$\label{fig:H(n,1)Polyhedron}.}
\end{figure}

\begin{proof}[Proof of Theorem~\ref{mainthm:CHRspines}] (a) By Theorem~\ref{mainthm:CHRmanifoldgroup} we may assume that $(m,k)=1$ and either $2k\equiv 0$~mod~$n$ or $2(m-k)\equiv 0$~mod~$n$, in which case $\mathcal{G}_n(m,k)$ is not a 3-manifold spine by Lemma~\ref{lem:cyclicCHRnonspine}.

(b) If $(n,m)=(9,4)$ or $(9,7)$ then $\mathcal{H}(n,m)$ is not a 3-manifold spine by Lemma~\ref{lem:H94H97}, so assume $(n,m)\neq (9,4),(9,7)$.
By Theorem~\ref{mainthm:CHRmanifoldgroup} we may assume that $m=0$ or $m=1$ or $m=2$ or ($m=n-1$ and $n=3,5,7$ or $n\geq 4$ is even) or $(n,m)=(5,3)$ or $(7,4)$ or $m=n/2+1$ (where $n\geq 4$ is even).
In the case $m=0$ the presentation $\mathcal{H}(n,m)$ is not a 3-manifold spine by Lemma~\ref{lem:Hn0nonspine}. In the cases $m=1$ and $m=2$ the presentation $\mathcal{H}(n,m)$ is a 3-manifold spine by Lemma~\ref{lem:H(n,1)spine} and Theorem~\ref{thm:SieradskiSpine}, respectively.
In the case $m=n-1$ we have that $\mathcal{H}(n,m)$ is homeomorphic to $\mathcal{F}(2,n)$ so when $n\geq 4$ is even
then it is a 3-manifold spine by Theorem~\ref{thm:evenFibonacciSpine} and if $n=3,5$ or $7$ then it is a 3-manifold spine if and only if $n=3$ by~\cite[page~55]{CavicchioliSpaggiari}. In the cases $(n,m)=(5,3)$ or $(7,4)$ we have that $\mathcal{H}(n,m)$ is homeomorphic to $\mathcal{F}(2,5)$ or $\mathcal{F}(2,7)$, so again is not a 3-manifold spine. In the case $m=n/2+1$ we have that the presentation $\mathcal{H}(n,m)$ is not a 3-manifold spine by Lemma~\ref{lem:cyclicCHRnonspine}.
\end{proof}

\section{Fibonacci and Sieradski manifolds revisited}

In~\cite[Example~1.2]{KimVesnin97} the presentation $\mathcal{G}_m(x_0^{-1}x_{1}^{2}x_{2}^{-1}x_{1})$ of the Fibonacci group $F(2,2m)$ was obtained. In the same way we can obtain an alternative presentation of the Sieradski group $S(2,2m)$:
\begin{alignat*}{1}
S(2,2m)
&=\pres{x_i}{x_ix_{i+2}=x_{i+1}\ (0\leq i\leq 2m-1)}\\
&=\pres{x_{2j},x_{2j+1}}{x_{2j}x_{2j+2}=x_{2j+1},x_{2j+1}x_{2j+3}=x_{2j+2}\ (0\leq j\leq m-1)}\\
&=\pres{x_{2j}}{(x_{2j}x_{2j+2})(x_{2j+2}x_{2j+4})=x_{2j+2}\ (0\leq j\leq m-1)}\\
&=\pres{y_{j}}{(y_{j}y_{j+1})(y_{j+1}y_{j+2})=y_{j+1}\ (0\leq j\leq m-1)}\\
&=\mathcal{G}_m(y_0y_1^2y_2y_1^{-1}).
\end{alignat*}

\begin{maintheorem}\label{mainthm:alternativepresspines}
For each $m\geq 3$ the cyclic presentation $\mathcal{G}_m(x_0^{-1}x_{1}^{2}x_{2}^{-1}x_{1})$ of the Fibonacci group $F(2,2m)$ and the cyclic presentation $\mathcal{G}_m(x_0x_1^2x_2x_1^{-1})$ of the Sieradski group $S(2,2m)$ are spines of closed 3-manifolds.
\end{maintheorem}

\begin{proof}[Proof of Theorem~\ref{mainthm:alternativepresspines}]
Consider first the presentation $\mathcal{G}=\mathcal{G}_m(x_0^{-1}x_{1}^{2}x_{2}^{-1}x_{1})$ of the Fibonacci group $F(2,2m)$ and consider the polyhedron in Figure~\ref{fig:AltFibonacciPolyhedron} with face pairings as shown. We shall show that the identification of faces results in a 3-complex $M$ whose 2-skeleton is the presentation complex of $\mathcal{G}$, so $M$ has one $0$-cell, $m$ 1-cells, $m$ 2-cells and one $3$-cell, so is a manifold by~\cite[Theorem~I, Section~60]{SeifertThrelfall}. To this end note that all edges labelled $x_2$ are identified by the following cyclic of faces identifications:
\[ [N,v_1] \stackrel{F_1}{\longleftrightarrow} [u_2,S] \stackrel{F_2}{\longleftrightarrow} [w_2,v_2] \stackrel{F_1}{\longleftrightarrow} [w_1,w_2]  \stackrel{F_1}{\longleftrightarrow} [u_1,w_1] \stackrel{F_0}{\longleftrightarrow} [N,v_1]. \]
By symmetry, for each $0\leq i\leq m-1$, all edges labelled $x_i$ are similarly identified. The initial vertices of the edges labelled $x_i$, namely $u_i,u_{i-1},w_i,w_{i-1},N$ are identified and the terminal vertices of the edges labelled $x_i$, namely $v_i,v_{i-1},w_i,w_{i-1},S$ are identified. Therefore the resulting complex has one $0$-cell, $m$ $1$-cells, $m$ $2$-cells, and one $3$-cell, and so its 2-skeleton is the presentation complex of $\mathcal{G}$, as required.

Consider now the presentation $\mathcal{G}_n(x_0x_1^2x_2x_1^{-1})$ of the Sieradski group $S(2,2m)$ and consider the polyhedron in Figure~\ref{fig:AltSieradskiPolyhedron}. We argue as above. The edges labelled $x_3$ are identified by the following cycle of face identifications:
\[ [N,u_4] \stackrel{F_2}{\longleftrightarrow} [v_2,S] \stackrel{F_1}{\longleftrightarrow} [w_2,u_3] \stackrel{F_2}{\longleftrightarrow} [u_3,v_3]  \stackrel{F_2}{\longleftrightarrow} [v_3,w_3] \stackrel{F_3}{\longleftrightarrow} [N,u_4]. \]
By symmetry, for each $0\leq i\leq m-1$, all edges labelled $x_i$ are similarly identified. The initial vertices of the edges labelled $x_i$, namely $v_{i-1},w_{i-1},u_i,v_i,N$ are identified and the terminal vertices of the edges labelled $x_i$, namely $u_{i+1},u_i,v_i,w_i,S$ are identified.
\end{proof}


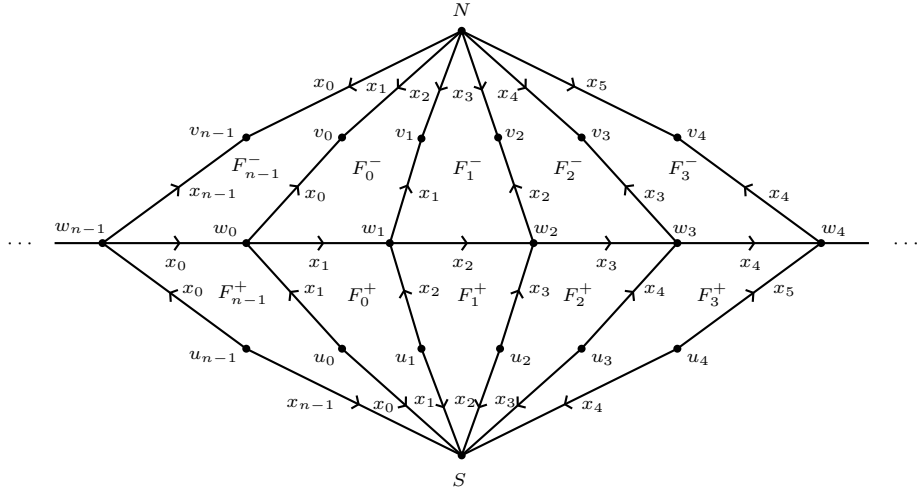
\begin{figure}
\vspace{-0.5cm}
\psset{xunit=0.63cm,yunit=0.7cm,algebraic=true,dimen=middle,dotstyle=o,dotsize=5pt 0,linewidth=0.8pt,arrowsize=3pt 2,arrowinset=0.25}
\begin{pspicture*}(-2.5,-5)(29.28,5)
\psline(0.,0.)(3.,0.)
\psline(1.605,0.)(1.5,-0.135)
\psline(1.605,0.)(1.5,0.135)
\psline(3.,0.)(6.,0.)
\psline(4.605,0.)(4.5,-0.135)
\psline(4.605,0.)(4.5,0.135)
\psline(6.,0.)(9.,0.)
\psline(7.605,0.)(7.5,-0.135)
\psline(7.605,0.)(7.5,0.135)
\psline(9.,0.)(12.,0.)
\psline(10.605,0.)(10.5,-0.135)
\psline(10.605,0.)(10.5,0.135)
\psline(12.,0.)(15.,0.)
\psline(13.605,0.)(13.5,-0.135)
\psline(13.605,0.)(13.5,0.135)
\psline(7.5,4.02)(3.,2.)
\psline(5.154208452700762,2.967000238767897)(5.194714592701582,3.133160560813306)
\psline(5.154208452700762,2.967000238767897)(5.305285407298417,2.8868394391866934)
\psline(7.5,4.02)(5.,2.)
\psline(6.16832846112781,2.9440093965912695)(6.165154938474491,3.1150062642642444)
\psline(6.16832846112781,2.9440093965912695)(6.334845061525511,2.9049937357357547)
\psline(7.5,4.02)(6.66,1.98)
\psline(7.040021270448299,2.9029087996601546)(6.955168456705913,3.0514012237093295)
\psline(7.040021270448299,2.9029087996601546)(7.204831543294086,2.94859877629067)
\psline(7.5,4.02)(8.26,2.)
\psline(7.916974571466961,2.911725481100975)(7.753647047129825,2.9624612652567657)
\psline(7.916974571466961,2.911725481100975)(8.006352952870175,3.0575387347432335)
\psline(7.5,4.02)(10.,2.)
\psline(8.83167153887219,2.9440093965912695)(8.665154938474489,2.9049937357357547)
\psline(8.83167153887219,2.9440093965912695)(8.83484506152551,3.1150062642642444)
\psline(7.5,4.02)(12.,2.)
\psline(9.845791547299239,2.967000238767897)(9.694714592701583,2.8868394391866934)
\psline(9.845791547299239,2.967000238767897)(9.805285407298419,3.133160560813306)
\psline(0.,0.)(3.,2.)
\psline(1.587365280905474,1.0582435206036493)(1.574884526490406,0.8876732102643915)
\psline(1.587365280905474,1.0582435206036493)(1.425115473509594,1.1123267897356093)
\psline(3.,0.)(5.,2.)
\psline(4.074246212024588,1.0742462120245873)(4.095459415460184,0.9045405845398158)
\psline(4.074246212024588,1.0742462120245873)(3.904540584539816,1.095459415460184)
\psline(6.,0.)(6.66,1.98)
\psline(6.363203915431768,1.089611746295304)(6.45807224523682,0.947309251587727)
\psline(6.363203915431768,1.089611746295304)(6.20192775476318,1.0326907484122738)
\psline(9.,0.)(8.26,2.)
\psline(8.593564067411835,1.0984754934815293)(8.756611348761966,1.0468461990419278)
\psline(8.593564067411835,1.0984754934815293)(8.503388651238035,0.953153800958073)
\psline(12.,0.)(10.,2.)
\psline(10.925753787975413,1.0742462120245873)(11.095459415460184,1.095459415460184)
\psline(10.925753787975413,1.0742462120245873)(10.904540584539818,0.9045405845398158)
\psline(15.,0.)(12.,2.)
\psline(13.412634719094527,1.0582435206036493)(13.574884526490406,1.1123267897356093)
\psline(13.412634719094527,1.0582435206036493)(13.425115473509592,0.8876732102643915)
\psline(3.,-2.)(7.5,-4.02)
\psline(5.345791547299239,-3.052999761232101)(5.194714592701583,-3.133160560813305)
\psline(5.345791547299239,-3.052999761232101)(5.305285407298418,-2.886839439186693)
\psline(5.,-2.)(7.5,-4.02)
\psline(6.3316715388721905,-3.075990603408728)(6.165154938474489,-3.1150062642642427)
\psline(6.3316715388721905,-3.075990603408728)(6.33484506152551,-2.904993735735753)
\psline(6.66,-2.)(7.5,-4.02)
\psline(7.1203164457271715,-3.1069514528200983)(6.955348132088443,-3.061835430220646)
\psline(7.1203164457271715,-3.1069514528200983)(7.20465186791156,-2.9581645697793504)
\psline(8.3,-2.)(7.5,-4.02)
\psline(7.86133750425491,-3.1076228017563547)(7.774484969170401,-2.960291076899167)
\psline(7.86133750425491,-3.1076228017563547)(8.025515030829602,-3.05970892310083)
\psline(10.,-2.)(7.5,-4.02)
\psline(8.66832846112781,-3.075990603408728)(8.66515493847449,-2.904993735735753)
\psline(8.66832846112781,-3.075990603408728)(8.834845061525511,-3.1150062642642427)
\psline(3.,-2.)(0.,0.)
\psline(1.4126347190945263,-0.9417564793963503)(1.574884526490406,-0.8876732102643904)
\psline(1.4126347190945263,-0.9417564793963503)(1.425115473509594,-1.1123267897356082)
\psline(5.,-2.)(3.,0.)
\psline(3.9257537879754127,-0.9257537879754123)(4.095459415460184,-0.9045405845398158)
\psline(3.9257537879754127,-0.9257537879754123)(3.904540584539816,-1.095459415460184)
\psline(6.66,-2.)(6.,0.)
\psline(6.297095367308591,-0.9002889918442132)(6.458199867628869,-0.9576940436824735)
\psline(6.297095367308591,-0.9002889918442132)(6.201800132371131,-1.0423059563175274)
\psline(8.3,-2.)(9.,0.)
\psline(8.68468679459645,-0.9008948725815684)(8.777420878109412,-1.0445973073382946)
\psline(8.68468679459645,-0.9008948725815684)(8.522579121890587,-0.9554026926617063)
\psline(10.,-2.)(12.,0.)
\psline(11.074246212024587,-0.9257537879754123)(11.095459415460182,-1.095459415460184)
\psline(11.074246212024587,-0.9257537879754123)(10.904540584539816,-0.9045405845398158)
\psline(-1.,0.)(0.,0.)
\psline(15.,0.)(16.,0.)
\psline(12.,-2.)(15.,0.)
\psline(13.587365280905473,-0.9417564793963503)(13.574884526490408,-1.1123267897356082)
\psline(13.587365280905473,-0.9417564793963503)(13.425115473509596,-0.8876732102643904)
\psline(12.,-2.)(7.5,-4.02)
\psline(9.654208452700761,-3.052999761232101)(9.694714592701583,-2.886839439186693)
\psline(9.654208452700761,-3.052999761232101)(9.805285407298419,-3.133160560813305)
\begin{scriptsize}
\psdots[dotsize=3pt,dotstyle=*,linecolor=black](0.,0.)
\psdots[dotsize=3pt,dotstyle=*,linecolor=black](3.,0.)
\rput[bl](1.3,-0.5){$x_{0}$}
\psdots[dotsize=3pt,dotstyle=*,linecolor=black](6.,0.)
\rput[bl](5.4,0.14){\black{$w_1$}}
\rput[bl](2.3,0.14){\black{$w_0$}}
\rput[bl](-1.0,0.14){\black{$w_{n-1}$}}
\rput[bl](11.9,0.14){\black{$w_3$}}
\rput[bl](15.0,0.14){\black{$w_4$}}
\rput[bl](4.3,-0.5){$x_1$}
\psdots[dotsize=3pt,dotstyle=*,linecolor=black](9.,0.)
\rput[bl](9.0,0.14){\black{$w_2$}}
\rput[bl](7.3,-0.5){$x_2$}
\psdots[dotsize=3pt,dotstyle=*,linecolor=black](12.,0.)
\rput[bl](10.3,-0.5){$x_3$}
\psdots[dotsize=3pt,dotstyle=*,linecolor=black](15.,0.)
\rput[bl](13.3,-0.5){$x_4$}
\psdots[dotsize=3pt,dotstyle=*,linecolor=black](7.5,4.02)
\rput[bl](7.3,4.3){\black{$N$}}
\psdots[dotsize=3pt,dotstyle=*,linecolor=black](3.,2.)
\rput[bl](4.4,2.9){$x_0$}
\psdots[dotsize=3pt,dotstyle=*,linecolor=black](5.,2.)
\rput[bl](5.5,2.84){$x_1$}
\psdots[dotsize=3pt,dotstyle=*,linecolor=black](6.66,1.98)
\rput[bl](6.1,1.96){\black{$v_1$}}
\rput[bl](6.4,2.7){$x_2$}
\psdots[dotsize=3pt,dotstyle=*,linecolor=black](8.26,2.)
\rput[bl](8.4,1.96){\black{$v_2$}}
\rput[bl](1.8,1.96){\black{$v_{n-1}$}}
\rput[bl](4.4,1.96){\black{$v_0$}}
\rput[bl](10.2,1.96){\black{$v_3$}}
\rput[bl](12.2,1.96){\black{$v_4$}}
\rput[bl](7.3,2.7){$x_3$}
\psdots[dotsize=3pt,dotstyle=*,linecolor=black](10.,2.)
\rput[bl](8.24,2.7){$x_4$}
\psdots[dotsize=3pt,dotstyle=*,linecolor=black](12.,2.)
\rput[bl](10.1,2.9){$x_5$}
\rput[bl](1.8,0.8){$x_{n-1}$}
\rput[bl](4.22,0.8){$x_0$}
\rput[bl](6.6,0.8){$x_1$}
\rput[bl](8.9,0.8){$x_2$}
\rput[bl](11.3,0.8){$x_3$}
\rput[bl](13.9,0.8){$x_4$}
\psdots[dotsize=3pt,dotstyle=*,linecolor=black](7.5,-4.02)
\rput[bl](7.3,-4.6){\black{$S$}}
\psdots[dotsize=3pt,dotstyle=*,linecolor=black](3.,-2.)
\psdots[dotsize=3pt,dotstyle=*,linecolor=black](5.,-2.)
\psdots[dotsize=3pt,dotstyle=*,linecolor=black](6.66,-2.)
\rput[bl](6.1,-2.3){\black{$u_1$}}
\psdots[dotsize=3pt,dotstyle=*,linecolor=black](8.3,-2.)
\rput[bl](8.5,-2.3){\black{$u_2$}}
\rput[bl](4.4,-2.3){\black{$u_0$}}
\rput[bl](10.2,-2.3){\black{$u_3$}}
\rput[bl](12.2,-2.3){\black{$u_4$}}
\rput[bl](1.8,-2.3){\black{$u_{n-1}$}}
\psdots[dotsize=3pt,dotstyle=*,linecolor=black](10.,-2.)
\rput[bl](3.84,-3.2){$x_{n-1}$}
\rput[bl](5.64,-3.2){$x_0$}
\rput[bl](6.5,-3.1){$x_1$}
\rput[bl](7.34,-3.1){$x_2$}
\rput[bl](8.18,-3.1){$x_3$}
\rput[bl](10.0,-3.2){$x_4$}
\rput[bl](1.66,-1.0){$x_0$}
\rput[bl](4.2,-1.0){$x_1$}
\rput[bl](6.6,-1.0){$x_2$}
\rput[bl](8.9,-1.0){$x_3$}
\rput[bl](11.3,-1.0){$x_4$}
\rput[bl](14.0,-1.0){$x_5$}
\rput[bl](-2.0,-0.1){\black{$\cdots$}}
\rput[bl](16.5,-0.1){\black{$\cdots$}}
\psdots[dotsize=3pt,dotstyle=*,linecolor=black](12.,-2.)
\rput[bl](2.68,1.2){\black{$F_{n-1}^-$}}
\rput[bl](5.2,1.2){\black{$F_0^-$}}
\rput[bl](7.3,1.2){\black{$F_1^-$}}
\rput[bl](9.4,1.2){\black{$F_2^-$}}
\rput[bl](11.8,1.2){\black{$F_3^-$}}
\rput[bl](12.4,-1.2){\black{$F_3^+$}}
\rput[bl](9.6,-1.2){\black{$F_2^+$}}
\rput[bl](7.4,-1.2){\black{$F_1^+$}}
\rput[bl](5.1,-1.2){\black{$F_0^+$}}
\rput[bl](2.4,-1.2){\black{$F_{n-1}^+$}}
\end{scriptsize}
\end{pspicture*}

\caption{Face pairing polyhedron for the presentation $\mathcal{G}_m(x_0^{-1}x_1^2x_2^{-1}x_1)$ of the Fibonacci group~$F(2,2m)$\label{fig:AltFibonacciPolyhedron}.}
\end{figure}

\begin{figure}
\vspace{-0.5cm}

\psset{xunit=0.63cm,yunit=0.7cm,algebraic=true,dimen=middle,dotstyle=o,dotsize=5pt 0,linewidth=0.8pt,arrowsize=3pt 2,arrowinset=0.25}
\begin{pspicture*}(-2,-5)(29.78,5)
\psline(0.,0.)(1.,0.)
\psline(0.605,0.)(0.5,-0.135)
\psline(0.605,0.)(0.5,0.135)
\psline(1.,0.)(2.,0.)
\psline(1.605,0.)(1.5,-0.135)
\psline(1.605,0.)(1.5,0.135)
\psline(2.,0.)(3.,0.)
\psline(2.605,0.)(2.5,-0.135)
\psline(2.605,0.)(2.5,0.135)
\psline(3.,0.)(4.,0.)
\psline(3.605,0.)(3.5,-0.135)
\psline(3.605,0.)(3.5,0.135)
\psline(4.,0.)(5.,0.)
\psline(4.605,0.)(4.5,-0.135)
\psline(4.605,0.)(4.5,0.135)
\psline(5.,0.)(6.,0.)
\psline(5.605,0.)(5.5,-0.135)
\psline(5.605,0.)(5.5,0.135)
\psline(6.,0.)(7.,0.)
\psline(6.605,0.)(6.5,-0.135)
\psline(6.605,0.)(6.5,0.135)
\psline(7.,0.)(8.,0.)
\psline(7.605,0.)(7.5,-0.135)
\psline(7.605,0.)(7.5,0.135)
\psline(8.,0.)(9.,0.)
\psline(8.605,0.)(8.5,-0.135)
\psline(8.605,0.)(8.5,0.135)
\psline(9.,0.)(10.,0.)
\psline(9.605,0.)(9.5,-0.135)
\psline(9.605,0.)(9.5,0.135)
\psline(10.,0.)(11.,0.)
\psline(10.605,0.)(10.5,-0.135)
\psline(10.605,0.)(10.5,0.135)
\psline(11.,0.)(12.,0.)
\psline(11.605,0.)(11.5,-0.135)
\psline(11.605,0.)(11.5,0.135)
\psline(12.,0.)(13.,0.)
\psline(12.605,0.)(12.5,-0.135)
\psline(12.605,0.)(12.5,0.135)
\psline(13.,0.)(14.,0.)
\psline(13.605,0.)(13.5,-0.135)
\psline(13.605,0.)(13.5,0.135)
\psline(14.,0.)(15.,0.)
\psline(14.605,0.)(14.5,-0.135)
\psline(14.605,0.)(14.5,0.135)
\psline(8.,4.)(0.,0.)
\psline(3.9060851449450094,1.953042572472505)(3.9396261646075064,2.1207476707849895)
\psline(3.9060851449450094,1.953042572472505)(4.060373835392495,1.879252329215012)
\psline(8.,4.)(3.,0.)
\psline(5.418008775008482,1.9344070200067855)(5.415666168580152,2.1054172892748095)
\psline(5.418008775008482,1.9344070200067855)(5.584333831419847,1.8945827107251914)
\psline(8.,4.)(6.,0.)
\psline(6.953042572472505,1.9060851449450091)(6.879252329215013,2.0603738353924945)
\psline(6.953042572472505,1.9060851449450091)(7.120747670784988,1.939626164607506)
\psline(8.,4.)(9.,0.)
\psline(8.525466240628814,1.8981350374847408)(8.36903076248038,1.9672576906200958)
\psline(8.525466240628814,1.8981350374847408)(8.630969237519619,2.0327423093799055)
\psline(8.,4.)(12.,0.)
\psline(10.074246212024587,1.9257537879754127)(9.904540584539816,1.9045405845398162)
\psline(10.074246212024587,1.9257537879754127)(10.095459415460182,2.0954594154601844)
\psline(8.,4.)(15.,0.)
\psline(11.59116552992307,1.9479054114725327)(11.433021243321829,1.882787175813198)
\psline(11.59116552992307,1.9479054114725327)(11.566978756678173,2.1172128241868022)
\psline(1.,0.)(9.,-4.)
\psline(5.093914855054991,-2.0469574275274947)(4.939626164607505,-2.120747670784988)
\psline(5.093914855054991,-2.0469574275274947)(5.060373835392494,-1.8792523292150103)
\psline(4.,0.)(9.,-4.)
\psline(6.581991224991518,-2.0655929799932133)(6.415666168580153,-2.105417289274808)
\psline(6.581991224991518,-2.0655929799932133)(6.584333831419847,-1.8945827107251902)
\psline(7.,0.)(9.,-4.)
\psline(8.046957427527495,-2.09391485505499)(7.879252329215013,-2.060373835392493)
\psline(8.046957427527495,-2.09391485505499)(8.120747670784988,-1.9396261646075041)
\psline(10.,0.)(9.,-4.)
\psline(9.474533759371186,-2.1018649625152594)(9.369030762480381,-1.9672576906200951)
\psline(9.474533759371186,-2.1018649625152594)(9.63096923751962,-2.0327423093799046)
\psline(13.,0.)(9.,-4.)
\psline(10.925753787975413,-2.074246212024587)(10.904540584539818,-1.9045405845398151)
\psline(10.925753787975413,-2.074246212024587)(11.095459415460184,-2.0954594154601835)
\psline(16.,0.)(9.,-4.)
\psline(12.40883447007693,-2.052094588527466)(12.433021243321829,-1.8827871758131964)
\psline(12.40883447007693,-2.052094588527466)(12.566978756678173,-2.117212824186801)
\psline(15.,0.)(16.,0.)
\psline(15.605,0.)(15.5,-0.135)
\psline(15.605,0.)(15.5,0.135)
\psline(16.,0.)(17.,0.)
\psline(-1.,0.)(0.,0.)
\begin{scriptsize}
\psdots[dotsize=3pt 0,dotsize=3pt,dotstyle=*,linecolor=black](0.,0.)
\psdots[dotsize=3pt,dotstyle=*,linecolor=black](1.,0.)
\rput[bl](0.3,-0.5){$x_{0}$}
\psdots[dotsize=3pt,dotstyle=*,linecolor=black](2.,0.)
\rput[bl](1.3,0.2){$x_0$}
\psdots[dotsize=3pt,dotstyle=*,linecolor=black](3.,0.)
\rput[bl](2.3,0.2){$x_1$}
\psdots[dotsize=3pt,dotstyle=*,linecolor=black](4.,0.)
\rput[bl](3.3,-0.5){$x_1$}
\psdots[dotsize=3pt,dotstyle=*,linecolor=black](5.,0.)
\rput[bl](4.5,-0.4){$x_1$}
\psdots[dotsize=3pt,dotstyle=*,linecolor=black](6.,0.)

\rput[bl](-0.3,-0.4){\black{$u_0$}}
\rput[bl](0.8,-0.4){\black{$v_0$}}
\rput[bl](1.8,-0.4){\black{$w_0$}}
\rput[bl](2.8,-0.4){\black{$u_1$}}
\rput[bl](3.8,0.1){\black{$v_1$}}
\rput[bl](4.9,0.1){\black{$w_1$}}
\rput[bl](5.9,-0.4){\black{$u_2$}}
\rput[bl](6.6,-0.4){\black{$v_2$}}
\rput[bl](7.8,-0.4){\black{$w_2$}}
\rput[bl](9.0,0.1){\black{$u_3$}}
\rput[bl](9.8,0.1){\black{$v_3$}}
\rput[bl](10.7,0.1){\black{$w_3$}}
\rput[bl](12.8,0.1){\black{$v_4$}}
\rput[bl](13.8,0.1){\black{$w_4$}}
\rput[bl](14.8,-0.4){\black{$u_5$}}
\rput[bl](15.8,0.1){\black{$v_5$}}
\rput[bl](5.3,-0.4){$x_2$}
\psdots[dotsize=3pt,dotstyle=*,linecolor=black](7.,0.)
\rput[bl](6.4,0.2){$x_{2}$}
\psdots[dotsize=3pt,dotstyle=*,linecolor=black](8.,0.)
\rput[bl](7.3,0.2){$x_2$}
\psdots[dotsize=3pt,dotstyle=*,linecolor=black](9.,0.)
\rput[bl](8.3,0.2){$x_3$}
\psdots[dotsize=3pt,dotstyle=*,linecolor=black](10.,0.)
\rput[bl](9.3,-0.5){$x_3$}
\psdots[dotsize=3pt,dotstyle=*,linecolor=black](11.,0.)
\rput[bl](10.3,-0.5){$x_3$}
\psdots[dotsize=3pt,dotstyle=*,linecolor=black](12.,0.)
\rput[bl](11.8,-0.4){\black{$u_4$}}
\rput[bl](11.3,-0.5){$x_4$}
\psdots[dotsize=3pt,dotstyle=*,linecolor=black](13.,0.)
\rput[bl](12.3,0.2){$x_4$}
\psdots[dotsize=3pt,dotstyle=*,linecolor=black](14.,0.)
\rput[bl](13.4,-0.5){$x_4$}
\psdots[dotsize=3pt,dotstyle=*,linecolor=black](15.,0.)
\rput[bl](14.4,-0.5){$x_5$}
\psdots[dotsize=3pt,dotstyle=*,linecolor=black](8.,4.)
\rput[bl](3.1,2.1){$x_{n-1}$}
\rput[bl](4.9,2.1){$x_0$}
\rput[bl](6.5,2.1){$x_{1}$}
\rput[bl](8.6,2.1){$x_2$}
\rput[bl](10.1,2.1){$x_3$}
\rput[bl](11.6,2.1){$x_4$}
\psdots[dotsize=3pt,dotstyle=*,linecolor=black](9.,-4.)
\rput[bl](7.8,4.2){\black{$N$}}
\rput[bl](8.8,-4.6){\black{$S$}}
\rput[bl](4.1,-2.1){$x_1$}
\rput[bl](5.8,-2.1){$x_2$}
\rput[bl](7.3,-2.1){$x_3$}
\rput[bl](9.7,-2.1){$x_4$}
\rput[bl](11.3,-2.1){$x_5$}
\psdots[dotsize=3pt,dotstyle=*,linecolor=black](16.,0.)
\rput[bl](13.0,-2.1){$x_6$}
\rput[bl](15.2,0.2){$x_5$}
\rput[bl](-2.0,-0.1){\black{$\cdots$}}
\rput[bl](17.5,-0.1){\black{$\cdots$}}
\rput[bl](3.08,1.0){\black{$F_{n-1}^-$}}
\rput[bl](5.22,1.0){\black{$F_{0}^-$}}
\rput[bl](7.50,1.0){\black{$F_1^-$}}
\rput[bl](9.44,1.0){\black{$F_2^-$}}
\rput[bl](11.5,1.0){\black{$F_3^-$}}
\rput[bl](4.08,-1.4){\black{$F_0^+$}}
\rput[bl](6.08,-1.4){\black{$F_1^+$}}
\rput[bl](8.44,-1.4){\black{$F_2^+$}}
\rput[bl](10.48,-1.4){\black{$F_3^+$}}
\rput[bl](12.8,-1.4){\black{$F_4^+$}}
\end{scriptsize}
\end{pspicture*}

\caption{Face pairing polyhedron for the presentation $\mathcal{G}_n(x_0x_1^2x_2x_1^{-1})$ of the Sieradski group~$S(2,2m)$\label{fig:AltSieradskiPolyhedron}.}
\end{figure}
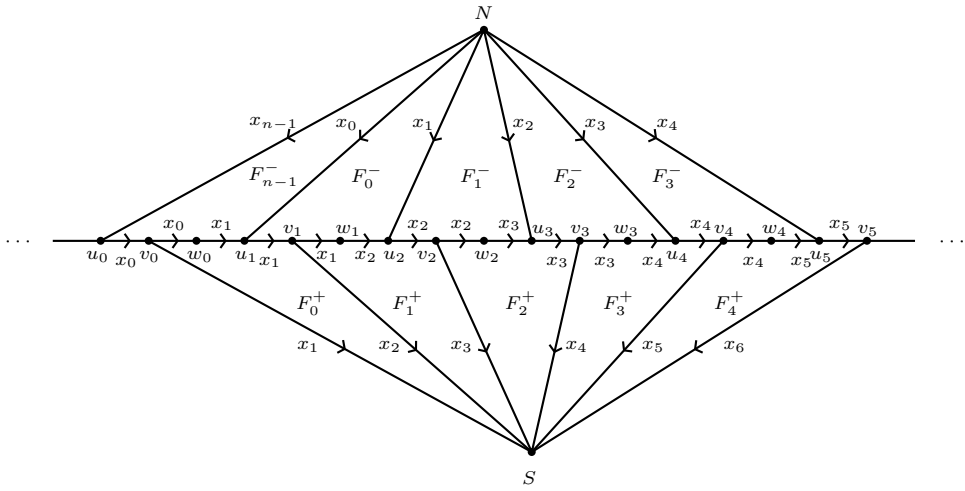

Observe that the polyhedron given in Figure~\ref{fig:AltFibonacciPolyhedron} for the presentation \linebreak $\mathcal{G}_m(x_0^{-1}x_1^2x_2^{-1}x_1)$  of $F(2,2m)$ appears to be of a fundamentally different nature to the polyhedron for the presentation $\mathcal{F}(2,2m)=\mathcal{G}_m(x_0x_1x_2^{-1})$ given in~\cite[Figure~1]{HMK}: the former consists of $2m$ pentagons whereas the latter consists of $4m$ triangles. Similarly, the polyhedron given in Figure~\ref{fig:AltSieradskiPolyhedron} for the presentation $\mathcal{G}_n(x_0x_1^2x_2x_1^{-1})$ of $S(2,2m)$
appears to be of a fundamentally different nature to the polyhedron for the presentation $\mathcal{S}(2,2m)=\mathcal{G}_{2m}(x_0x_2x_1^{-1})$ given in~\cite[Figure~8]{Sieradski} or~\cite[Figure~1]{CHK}: again, the former consists of $2m$ pentagons, while the latter consists of $4m$ triangles.

The alternative Fibonacci presentation $\mathcal{G}_n(x_0^{-1}x_{1}^{2}x_{2}^{-1}x_{1})$ is already known to be a 3-manifold spine, as it appears in line~6 of~\cite[Table~1]{Dunwoody}; however, our proof gives an explicit construction of the face-pairing polyhedron. As observed in~\cite[Example~1.2]{KimVesnin97} the polynomial $1+t-t^2$ associated to the original Fibonacci presentation $\mathcal{G}_{2m}(x_0x_1x_2^{-1})$ is not an Alexander polynomial, but the polynomial $1-3t+t^2$ associated to the alternative Fibonacci presentation $\mathcal{G}_{m}(x_0^{-1}x_{1}^{2}x_{2}^{-1}x_{1})$ is the Alexander polynomial of the figure eight knot. For the Sieradski presentations, the polynomial $1-t+t^2$ for the presentation $\mathcal{G}_{2m}(x_0x_2x_1^{-1})$ is the Alexander polynomial of the trefoil knot, but the polynomial $1+t+t^2$ for the presentation $\mathcal{G}_m(x_0x_1^2x_2x_1^{-1})$ is not an Alexander polynomial~\cite{BurdeZieschang}.

\section*{Acknowledgements}

We thank Bill Bogley for helpful comments on a draft of this article.

\bigskip

  \textsc{Department of Mathematics and Maxwell Institute for Mathematical Sciences, Heriot-Watt University, Edinburgh
EH14 4AS, UK.}\par\nopagebreak
  \textit{E-mail address}, \texttt{J.Howie@hw.ac.uk}

  \medskip

  \textsc{Department of Mathematical Sciences, University of Essex, Wivenhoe Park, Colchester, Essex CO4 3SQ, UK.}\par\nopagebreak
  \textit{E-mail address}, \texttt{Gerald.Williams@essex.ac.uk}

\end{document}